\documentclass[a4paper,twoside,10pt]{amsart}
\usepackage{amsmath,amsfonts,amssymb,amsthm,mathtools}
\usepackage{todonotes}

\newtheorem{lemma}{Lemma}[section]

\usepackage{color,graphicx}
\usepackage[colorlinks=true,pdfstartview=FitH,pdfview=FitH]{hyperref}
\numberwithin{equation}{section}

\newcommand\mytop[2]{\genfrac{}{}{0pt}{}{#1}{#2}}

\def\qphi#1#2#3#4#5#6{{{}_{#1}\phi_{#2}}\!\left(\genfrac{}{}{0pt}{}{#3}{#4};#5;#6\right)}
\def\expe{{\rm e}}
\def\erfc{\operatorname{erfc}}

\def\iunit{{\rm i}}
\def\eps{\varepsilon}
\def\id{\,{\rm d}}
\def\bigO{{\mathcal O}}
\def\ifrac#1#2{\textstyle\frac{{#1}}{{#2}}\displaystyle}

\def\qP1{q{\rm P}_{\rm I}}
\def\P1{{\rm P}_{\rm I}}
\def\w1{w_{\rm I}}
\def\v1{v_{\rm I}}
\def\b1{B_{\rm I}}
\def\qpr#1#2#3{\left(#1;#2\right)_{#3}}
\def\Eq#1{{\rm E}_q\left(#1\right)}

\title[Exponentially-improved asymptotics for $q$-difference equations]{Exponentially-improved asymptotics for $q$-difference equations: ${}_2\phi_0$ and $\qP1$}
\author{Nalini Joshi}
\address{School of Mathematics and Statistics F07, The University of Sydney, New South Wales 2006, Australia, ORCID ID: 0000-0001-7504-4444}
\email{nalini.joshi@sydney.edu.au}
\author{Adri Olde Daalhuis}
\address{School of Mathematics, University of Edinburgh, Peter Guthrie Tait Road, Edinburgh EH9 3FD UK}
\email{a.oldedaalhuis@ed.ac.uk}
\date{\today}
\dedicatory{Dedicated to Tom Koornwinder on the occasion of his 80th birthday}


\begin{document}
\begin{abstract}
Usually when solving differential or difference equations via series solutions one encounters
divergent series in which the coefficients grow like a factorial. Surprisingly, in the $q$-world
the $n$th coefficient is often of the size $q^{-\frac12 n(n-1)}$, in which $q\in(0,1)$ is fixed.
Hence, the divergence is much stronger, and one has to introduce alternative Borel and Laplace
transforms to make sense of these formal series. We will discuss exponentially-improved asymptotics
for the basic hypergeometric function ${}_2\phi_0$ and for solutions of the $q$-difference first 
Painlev\'e equation $\qP1$. These are optimal truncated expansions, and re-expansions in terms of new 
$q$-hyperterminant functions. The re-expansions do incorporate the Stokes phenomena.
\end{abstract}

\maketitle

\section{Introduction and summary}
In solving differential or difference equations via series solutions one typically encounters Gevrey 1 divergent series
of the form $\hat{w}(\eps)=\sum_n a_n\eps^n$ in which $a_n\sim K\Gamma(n+\alpha)/F^{n+\alpha}$ as $n\to\infty$. Hence, the growth of the coefficients is of the form `factorial over power'. 
To make
sense of these formal series one introduces the Borel transform via $B(t)=\sum_n\frac{a_n}{\Gamma(n+\alpha)} t^n$. This
infinite series has a finite radius of convergence and the singularities in the finite Borel plane (the complex $t$-plane)
will cause the Stokes-phenomenon, that is, the birth of exponentially-small terms when Stokes curves are crossed. 
Each of these singularities is linked to its own exponentially-small scale. 
To return to the $\eps$-plane one
uses the Laplace transform $\int_0^\infty \expe^{-t/\eps}B(t)\id t$.
For more details see, for example, \cite{Mitschi2016}.

In solving $q$-difference equations, but also in the definition of the basic hypergeometric functions, we will encounter
a very different type of divergence. We will have an extra parameter $q$ and in this article we will assume that $|q|<1$.
In \S\ref{S:SpecialFunct} we will introduce all the special notation.
The typical growth that we will encounter is $a_n\sim Kq^{-\frac12 n(n-1)}/F^n$, as $n\to\infty$. Hence, the divergence
is much stronger, and as a consequence we have to use different Borel and Laplace transforms. 
They are introduced in \S\ref{S:qBL}.
The $q$-Borel transform
is, obviously, $B_q(t)=\sum_n a_n q^{\frac12 n(n-1)} t^n$. In the interesting cases it will have a finite radius of
convergence, and the singularities in the Borel plane will still cause the Stokes phenomenon, but now all the finite
singularities are linked to the same exponentially-small scale! For the corresponding $q$-Laplace transform we have
to introduce an alternative for $\expe^{-t/\eps}$.
The growth of this alternative exponential function $\Eq{t}$
is, in fact, in between algebraic and exponential, and this is the reason that all the Borel singularities
are linked to the same exponentially-small scale.

In Gevrey asymptotics an optimal truncated asymptotic expansion can be re-expanded
in terms of hyperterminants. These level 1 hyperasymptotic approximations determine
uniquely the solutions not only numerically, but also asymptotically. The first
hyperterminant is the simplest function with a Stokes phenomenon in its
asymptotic tail, and as a consequence, the level 1 hyperasymptotic expansion
incorporates the Stokes phenomenon of our original problem. See \cite{OD1998}.

In \S\ref{Sect:qhyperterminant} we introduce a $q$ version of the 
traditional hyperterminant function.
It will be the simplest function incorporating the $q$-Stokes phenomenon
in its asymptotic tail. We will
derive many of its properties, because this $q$-hyperterminant is the main
building block in the re-expansions of the optimal truncated asymptotic expansions.
One of our main results is a complete uniform asymptotic approximation 
of the $q$-hyperterminant in terms of the error functions $\erfc$. 
In this way we demonstrate that the $q$-Stokes phenomenon also has
`Berry-smoothing'. See \cite{Berry1988}.

Basic hypergeometric functions $\qphi{r}{s}{\mathbf{a}}{\mathbf{b}}{q}{z}$, with 
$r>s+1$, are generally divergent with the $q$-growth mentioned above. 
As a first application we will study in \S\ref{S:2phi0} the formal series $\qphi{2}{0}{a,b}{-}{q}{z}$
and show that the singularities in its $q$-Borel plane are just
simple poles. We will use this information to obtain exponentially improved
asymptotic approximations.

Like the classical first Painlev\'e equation, which appears widely as a model in physics, every solution of its $q$-difference version 
\begin{equation}\label{qP1}
 \qP1 \ :\quad   w(qz)w^2(z)w(z/q)=w(z)-z,\qquad 0<|q|<1,
\end{equation}
is a transcendental function \cite{N10}. This $q$-difference equation appears in
Sakai's list of discrete Painlev\'e equations  \cite{S01}, with an initial-value 
space of type $A_7^{(1)}$ and symmetry group $A_1^{(1)}$, and is conjectured to 
arise as a model in topological string/spectral theory \cite{BGT19}. Yet, hardly
anything is known about the asymptotic behaviours of its solutions near the 
singularities  of the equation, with one exception \cite{J15}. 
In \S\S\ref{Sect:qP1}--\ref{Sect:PolesZeros} we will study, as a second application,
the solutions of 
$\qP1$ with the property $w(z)\sim z$ as $|z|\to 0$. Its divergent asymptotic
expansion has again the same $q$-growth. We locate the singularities in the
$q$-Borel plane (again simple poles), and we compute the corresponding
Stokes multipliers. It is surprising that the dominant Stokes multiplier
seems to be $K_0(q)=\qpr{q}{q}{\infty}^{-2}$. The other Stokes multipliers can be
obtained via a transseries analysis. All of them can be expressed in terms
of $K_0(q)$. These Stokes multipliers are the
coefficients in the $q$-Stokes phenomenon and in the level 1 re-expansions in terms of 
our $q$-hyperterminant.

Two open problems are the following. (1) For the $q$-hyperterminant we give in
\S\ref{Sect:qhyperterminant} many of its
properties and even a full uniform asymptotic expansion across the Stokes curve, but
we still need an efficient method to numerically compute this function in the full
complex plane. 

(2) We still need a rigorous proof for the observation that $K_0(q)=\qpr{q}{q}{\infty}^{-2}$.
In appendix \ref{AppRiccati} we do show that a similar identity holds for 
the Stokes multiplier of a similar but slightly simpler $q$-difference equation.
We include initial steps that might help us for the Stokes multiplier of $\qP1$.


\section{$q$-special functions}\label{S:SpecialFunct}
As mentioned in the introduction we will assume that $|q|<1$, and we use the notation
$\qpr{a}{q}{\infty}=\prod_{n=0}^\infty \left(1-aq^n\right)$, and 
$\qpr{a}{q}{\nu}=\qpr{a}{q}{\infty}/\qpr{aq^\nu}{q}{\infty}$, $a,\nu\in\mathbb{C}$. 
Thus for $n$ a positive integer we have $\qpr{a}{q}{n}=(1-a)(1-aq)\cdots(1-aq^{n-1})$.

There are two $q$-exponential functions 
\begin{equation}\label{littleQexp}
    e_q(z)=\frac1{\qpr{z}{q}{\infty}}=\sum_{n=0}^\infty \frac{1}{\qpr{q}{q}{n}}z^n,
    \qquad |z|<1,
\end{equation}
and
\begin{equation}\label{biqQexp}
    E_q(z)=\qpr{-z}{q}{\infty}=\sum_{n=0}^\infty \frac{q^{\frac12n(n-1)}}{\qpr{q}{q}{n}}z^n,
    \qquad z\in\mathbb{C}.
\end{equation}
Formula \cite[(3.11)]{OD94} gives us for $z>1$ and $0<q<1$
\begin{equation}\label{ExpQlargez}
    \qpr{-z}{q}{\infty}=\frac{q^{-\frac{1}{12}}\sqrt{z}}{\qpr{-q/z}{q}{\infty}}
    \exp\left(\frac{-1}{\ln q}\left(\ifrac12\ln^2z+\ifrac16\pi^2\right)
    +\sum_{n=1}^\infty\frac{\cos\left(2\pi n\frac{\ln z}{\ln q}\right)}{n%
    \sinh\left(2\pi^2n/\ln q\right)}
    \right).
\end{equation}
This can be seen as an asymptotic approximation as $z\to +\infty$, with the observation
that the infinite series in \eqref{ExpQlargez} represents a $q$-periodic function.
Hence, the growth of $\qpr{-z}{q}{\infty}$ lies in between algebraic and exponential growth.

The expansions in \eqref{littleQexp} and \eqref{biqQexp} are special cases of the $q$-binomial
theorem (see \cite[\href{http://dlmf.nist.gov/17.2.E37}{17.2.37}]{NIST:DLMF}) which reads
\begin{equation}\label{qbinomial}
\frac{\qpr{az}{q}{\infty}}{\qpr{z}{q}{\infty}}=\sum_{n=0}^\infty
\frac{\qpr{a}{q}{n}}{\qpr{q}{q}{n}}z^n,\qquad |z|<1.
\end{equation}

We will encounter divergent series when we introduce the basic hypergeometric
functions
\begin{equation}\label{qhypergeom}
\qphi{r}{s}{a_1,\ldots,a_r}{b_1,\ldots,b_s}{q}{z}
=\sum_{n=0}^\infty
\frac{\qpr{a_1}{q}{n}\cdots\qpr{a_r}{q}{n}}{%
\qpr{b_1}{q}{n}\cdots\qpr{b_s}{q}{n}\qpr{q}{q}{n}}
\left(\left(-1\right)^nq^{\frac12n(n-1)}\right)^{1+s-r}z^n.
\end{equation}
Note that in this representation the most important factor for convergence
is the $q^{\frac12n(n-1)}$. The infinite series converges for all $z$ when
$r<s+1$, and for $|z|<1$ when $r=s+1$. In the case when $r>s+1$ the 
series diverges and we have to make sense of this series via the
$q$-Borel-Laplace transform introduced in the next section.

\section{The $q$-Borel-Laplace transform}\label{S:qBL}
Poincar\'e divergent series are typical of the form $\sum n! z^n$. We will see in this article
that in the $q$-world we encounter divergent series of the form
$\sum q^{-\frac12n(n-1)} z^n$. Hence, the divergence is much stronger. Let 
$\hat{w}(z)=\sum_{n=0}^\infty a_n z^n$ be a formal series with 
$|a_n|\leq \left|L^n q^{-\frac12n(n-1)}\right|$,
for large $n$. Then the obvious $q$-Borel transform is 
$B_q(t)=\sum_{n=0}^\infty a_n q^{\frac12n(n-1)}t^n$. This infinite series will typically
converge for $|t|<\widetilde{L}$, and ideally we are able to obtain information about
all the singularities of the analytical continuation of $B_q(t)$ in the complex
$t$-plane. Via a $q$-Laplace transform we will go back to the $z$-plane and obtain
a $q$-Borel-Laplace transform of formal series $\hat{w}(z)$.

In the literature there are many versions of the $q$-Laplace transform.
See for example \cite{Tahara2017}.
We will consider only the ones that are in terms of normal integrals.
Two of them are defined in \cite{Zhang2000} and \cite{Zhang2002}. We will have
to introduce the functions
\begin{equation}\label{newexp}
    \Eq{\tau}=\exp\left(\frac{\left(\ln(\tau/\sqrt{q})\right)^2}{2\ln q}\right)
    =q^{\frac18}\tau^{-\frac12}\exp\left(\frac{\ln^2 \tau}{2\ln q}\right),
\end{equation}
and the theta function with base $q$ is
\begin{equation}\label{theta}
    \theta_q\left(\tau\right)=\sum_{n=-\infty}^\infty
    q^{\frac12 n(n-1)}\tau^n=(q,-\tau,-q/\tau;q)_\infty.
\end{equation}
Both $F(\tau)=\Eq{\tau}$ and $F(\tau)=1/\theta_q\left(\tau\right)$ have
the properties
\begin{equation}\label{thetaprop}
    F\left(\tau^{-1}\right)=\tau F\left(\tau\right),
    \qquad
    F\left(q^n\tau\right)=\tau^{n}q^{\frac12 n(n-1)}F\left(\tau\right).
\end{equation}
In the case of $F(\tau)=\Eq{\tau}$ the second property holds for all $n\in\mathbb{C}$,
and in the case of $F(\tau)=1/\theta_q\left(\tau\right)$ it holds 
for all $n\in\mathbb{Z}$. We obtain
\begin{gather}
\begin{split}\label{theta0}
    \int_0^{\infty}\frac{1}{\theta_q\left(\tau\right)}\frac{\id \tau}{\tau}
    &=\sum_{n\in\mathbb{Z}} \int_{q^{n+1}}^{q^n}\frac{1}{\theta_q\left(\tau\right)}\frac{\id \tau}{\tau}
    =\sum_{n\in\mathbb{Z}} \int_q^1\frac{1}{\theta_q\left(q^n t\right)}\frac{\id t}{t}\\
    &=\sum_{n\in\mathbb{Z}} \int_q^1\frac{t^nq^{\frac12 n(n-1)}}{\theta_q\left( t\right)}\frac{\id t}{t}
    =\int_q^1\frac{\id t}{t}=-\ln q,
\end{split}
\end{gather}
and more importantly
\begin{equation}\label{thetan}
    C_q\int_0^{\infty}t^{n-1}\Eq{t/z}\,\id t=
    \frac{-1}{\ln q}\int_0^{\infty}\frac{t^{n-1}}{\theta_q\left(t/z\right)}\id t
    =q^{-\frac12 n(n-1)} z^n,
\end{equation}
in which
\begin{equation}\label{Cq}
    C_q=\frac1{\sqrt{-2\pi\ln q}}.
\end{equation}

From \eqref{thetaprop} we obtain that the function $\tau\mapsto\theta_q(\tau)\Eq{\tau}$ 
is $q$-periodic, which also follows from \eqref{ExpQlargez} because we have
\begin{equation}\label{qperiodic}
    \theta_q\left(\tau\right)\Eq{\tau}=q^{\frac1{24}}\left(q;q\right)_\infty
    \exp\left(\frac{-\pi^2}{6\ln q}
    +\sum_{n=1}^\infty\frac{\cos\left(2\pi n\frac{\ln\tau}{\ln q}\right)}{n%
    \sinh\left(2\pi^2n/\ln q\right)}
    \right).
\end{equation}
From the definition of $\Eq{\tau}$ it follows that its growth lies 
in between algebraic and exponential as $|\tau|\to \infty$
but also as $\tau\to 0$. The right-hand side of \eqref{qperiodic} is $q$-periodic. Hence,
$1/\theta_q\left(\tau\right)$ has the same growth property.

Now we have all the tools to introduce the $q$-Laplace transforms for a function $B_q(t)$:
\begin{equation}\label{qBorelLaplace1}
    w(z)=
    C_q\int_0^{\infty}
    B_q(t)\Eq{t/z}\frac{\id t}{t},
\end{equation}
\begin{equation}\label{qBorelLaplace2}
    \widetilde{w}(z)=\frac{-1}{\ln q}\int_0^{\infty}
    \frac{B_q(t)}{\theta_q\left(t/z\right)}\frac{\id t}{t}.
\end{equation}
It follows from \eqref{thetan} that both of these functions have the asymptotic
expansion $\sum_{n=0}^\infty a_n z^n$ as $z\to0$.

We will mainly use Laplace transform \eqref{qBorelLaplace1}, because it will give us a bit more freedom.
One reason is that the second property of \eqref{thetaprop} holds for all $n\in\mathbb{C}$.
However, most of the results below also hold if we use \eqref{qBorelLaplace2}.

\section{The $q$-hyperterminant function}\label{Sect:qhyperterminant}
The hyperterminant function that we introduce in this section is a $q$ version of the ones that we
use in hyperasymptotics. See for example \cite{OD1998}.

Let
\begin{equation}\label{qhyperterminant}
   F_q\left(z;\mytop{N}{\sigma}\right)=\sigma^{1-N}
    C_q\int_0^{\infty}
    \frac{t^{N-1}\Eq{t}}{t+\sigma z}\id t,\qquad\qquad |\arg(\sigma z)|<\pi.
\end{equation}
Then we have the following identities:
\begin{equation}\label{qhyper1}
   F_q\left(z;\mytop{N}{\sigma}\right)=
   \sum_{m=0}^{M-1}\frac{\left(-1\right)^m}{\sigma^{N+m} z^{m+1}}q^{-\frac12(N+m)(N+m-1)}
    +\left(-z\right)^{-M} F_q\left(z;\mytop{N+M}{\sigma}\right),
\end{equation}
that is, $z^{-1}\sigma^{-N} q^{-\frac12 N(N-1)}\qphi{2}{0}{0,q}{-}{q}{\frac1{\sigma zq^{N}}}$
and $F_q\left(z;\mytop{N}{\sigma}\right)$ have the same formal expansion.

\begin{equation}\label{qhyper2}
   F_q\left(z;\mytop{N}{\sigma q^{-j}}\right)=
   \frac{q^{\frac12j(j+1)}}{\sigma^j}F_q\left(z;\mytop{N-j}{\sigma}\right),
\end{equation}
and hence we can make the $\sigma$-entry always $1$
\begin{equation}\label{qhyper3}
   F_q\left(z;\mytop{N}{\sigma}\right)=
   \sigma^{\frac{\ln\sigma}{2\ln q}-\frac12}
   F_q\left(z;\mytop{N+\frac{\ln\sigma}{\ln q}}{1}\right),
\end{equation}
or make the $N$-entry 1 via
\begin{equation}\label{qhyper4}
   F_q\left(z;\mytop{N}{\sigma}\right)=
   \sigma^{1-N}q^{-\frac12(N-1)(N-2)}
   F_q\left(z;\mytop{1}{\sigma q^{N-1}}\right).
\end{equation}
Recurrence relations: $w(z)=F_q\left(z;\mytop{N}{\sigma}\right)$ does satisfy
\begin{equation}\label{qhyper5}
   q^{1-N}w(zq)+ z\sigma w(z)=\sigma^{1-N}q^{-\frac12 N(N-1)},
\end{equation}
and hence
\begin{equation}\label{qhyper6}
   q^{1-N} w(zq)
   +\left(\sigma z-q^{1-N}\right)w(z)-\frac{\sigma z}{q}w(z/q)=0.
\end{equation}
Note that $w(z)=z^{N-1}\Eq{-\sigma z}$ is an `elementary' solution of \eqref{qhyper6}. 
Compare the right-hand side of \eqref{qhyper8}. 
(This observation can be used to derive \eqref{qhyper5} from \eqref{qhyper6}.)
It follows from \eqref{qhyper1} and 
\eqref{newexp} that, unfortunately, $w(z)=F_q\left(z;\mytop{N}{\sigma}\right)$
is a dominant solution of \eqref{qhyper6} for both $z\to 0$ and $z\to\infty$.
Corresponding normalising conditions are
\begin{equation}\label{qhyper6a}
   \sum_{j=0}^\infty \frac{\left(\sigma z q^{N-1}\right)^{-j}}{(q;q)_j}
   F_q\left(zq^j;\mytop{N}{\sigma}\right)=
   \frac{\sigma^{-N}q^{-\frac12 N(N-1)}}{z\left(\frac{q^{1-N}}{\sigma z};q\right)_\infty},
   \qquad\qquad\left|\sigma z q^{N-1}\right|>1,
\end{equation}
\begin{equation}\label{qhyper6b}
   \sum_{j=0}^\infty \frac{\left(\sigma z q^{N-2}\right)^{j}}{(q;q)_j}
   F_q\left(zq^{-j};\mytop{N}{\sigma}\right)=
   \frac{\sigma^{1-N}q^{-\frac12(N-1)(N-2)}}{\left(\sigma zq^{N-1};q\right)_\infty},
   \qquad\qquad\left|\sigma z q^{N-1}\right|<1.
\end{equation}

Reflection formula
\begin{equation}\label{qhyper7}
   z\sigma F_q\left(z;\mytop{N}{\sigma}\right)=
   F_q\left(\frac{1}{z};\mytop{2-N}{\frac1{\sigma}}\right).
\end{equation}
Combining this reflection formula with \eqref{qhyper1} we identify
$F_q\left(1;\mytop{\frac12}{1}\right)=F_q\left(1;\mytop{\frac32}{1}\right)
=\frac12q^{\frac18}$.

Stokes phenomenon:
\begin{equation}\label{qhyper8}
   F_q\left(z\expe^{2\pi\iunit};\mytop{N}{\sigma}\right)
   -F_q\left(z;\mytop{N}{\sigma}\right)=
   2\pi\iunit C_q \left(-z\right)^{N-1}\Eq{-\sigma z}.
\end{equation}
Stokes smoothing: Taking $|z|$ large and $N=\left|\frac{\ln (z\sigma)}{\ln q}\right|+\bigO(1)$, then
\begin{gather}
\begin{split}\label{qhyper9}
   F_q\left(z;\mytop{N}{\sigma}\right)&\sim
   -\pi\iunit C_q\left(-z\right)^{N-1}\Eq{-\sigma z}\erfc\left(\iunit\tau_p N/\sqrt2\right)\\
   &\qquad +\frac{q^{-\frac12 N(N-1)}}{\sigma^N z}\left(\frac1{1+\frac{q^{\frac12-N}}{\sigma z}}
   -\frac1{\ln\left(-\sigma z q^{N-\frac12}\right)}\right),
\end{split}
\end{gather}
as $z\to \infty$, with $\tau_p=\frac{\ln\left(-\sigma zq^{N-\frac12}\right)}{N\sqrt{-\ln q}}$. Proof:

Below we use the substitution
$t=q^{\frac12-N}\expe^{\tau N\sqrt{-\ln q}}$ and obtain
\begin{gather}
\begin{split}\label{qhyper10}
   F_q\left(z;\mytop{N}{\sigma}\right)&=
   \frac{q^{\frac18}C_q}{\sigma^{N-1}}\int_0^{\infty}
    \frac{t^{N-\frac32}\expe^{\frac{\ln^2 t}{2\ln q}}}{t+\sigma z}\id t\\
    &=\frac{N\sqrt{-\ln q}C_q}{\sigma^N zq^{\frac12 N(N-1)}}\int_{-\infty}^{\infty}
    \frac{\expe^{-\frac12 N^2\tau^2}f(\tau)}{\tau_p-\tau}\id\tau\\
    &=\frac{N\sqrt{-\ln q}C_q}{\sigma^N zq^{\frac12 N(N-1)}}\left(\int_{-\infty}^{\infty}
    \frac{\expe^{-\frac12 N^2\tau^2}f(\tau_p)}{\tau_p-\tau}\id\tau+
    \int_{-\infty}^{\infty}
    \expe^{-\frac12 N^2\tau^2}g(\tau)\id\tau\right)\\
    &\sim  -\pi\iunit C_q\left(-z\right)^{N-1}\Eq{-\sigma z}\erfc\left(\iunit\tau_p N/\sqrt2\right)
    +\frac{g(0)}{\sigma^N zq^{\frac12 N(N-1)}},
\end{split}
\end{gather}
in which we have used
\begin{equation}\label{qhyper11}
   f(\tau)=\frac{\tau_p-\tau}{1-\exp((\tau-\tau_p)N\sqrt{-\ln q})},\qquad
   g(\tau)=\frac{f(\tau)-f(\tau_p)}{\tau_p-\tau},\qquad f(\tau_p)=\frac1{N\sqrt{-\ln q}},
\end{equation}
and the identity
\begin{equation}\label{qhyper12}
    \int_{-\infty}^{\infty}
    \frac{\expe^{-\frac12 N^2\tau^2}}{\tau_p-\tau}\id\tau=\pi\iunit \expe^{-\frac12 N^2\tau_p^2}
    \erfc\left(\iunit\tau_p N/\sqrt2\right).
\end{equation}
In fact it is easy to obtain more terms in the uniform asymptotic expansion 
via the Taylor-series of $g(\tau)$ about $\tau=0$, because
$g(\tau)=g_1(\tau)-g_2(\tau)$ with $g_1(\tau)=\frac{1}{1-\exp((\tau-\tau_p)N\sqrt{-\ln q})}$
and $g_2(\tau)=\frac{f(\tau_p)}{\tau_p-\tau}$. The Taylor series for $g_2(\tau)$ is just the
geometric progression, and for $g_1(\tau)$ we can use the nonlinear differential equation
$g_1'(\tau)=N\sqrt{-\ln q}\left(g_1^2(\tau)-g_1(\tau)\right)$. 
We write $\displaystyle g_1(\tau)=\sum_{n=0}^\infty c_n\tau^n$, and obtain 
\begin{equation}\label{unifromcoeff}
c_0=\frac1{1+\frac{q^{\frac12-N}}{\sigma z}},\qquad
    (n+1)c_{n+1}=N\sqrt{-\ln q}\left(\sum_{m=0}^n c_m c_{n-m}-c_n\right),
\end{equation}
$n=0,1,2,\ldots$. This results in the uniform asymptotic expansion
\begin{gather}
\begin{split}\label{qhyper13}
    F_q\left(z;\mytop{N}{\sigma}\right)&\sim
    -\pi\iunit C_q\left(-z\right)^{N-1}\Eq{-\sigma z}\erfc\left(\iunit\tau_p N/\sqrt2\right)\\
    &\qquad +\frac{q^{-\frac12 N(N-1)}}{\sigma^N z}
    \sum_{n=0}^\infty \frac{2^n\left(\frac12\right)_n}{
    N^{2n}}\left(c_{2n}-\frac{\tau_p^{-2n-1}}{N\sqrt{-\ln q}}\right).
\end{split}
\end{gather}

\section{Exponentially-improved asymptotics for ${}_2\phi_0$}\label{S:2phi0}
We did mention just below \eqref{qhypergeom} that for the definition of ${}_r\phi_s$,
when $r>s+1$, the $q$-Borel-Laplace transform is needed to make sense of these formal series.
In \cite{Adachi2019} it is demonstrated that these formal series are definitely $q$-Borel summable.
It is a very good source for more details for the general case.

In this section we give details for the case ${}_2\phi_0$. We do give the Stokes phenomenon
for this function, and via exponentially-improved asymptotics we do define this
$q$-Borel sum of ${}_2\phi_0$ both analytically and numerically.

We start with the formal series
\begin{equation}\label{2phi0}
    \hat{w}(z)=\qphi{2}{0}{a,b}{-}{q}{z}
    =\sum_{n=0}^\infty\frac{\qpr{a}{q}{n}\qpr{b}{q}{n}}{\qpr{q}{q}{n}}q^{-\frac12n(n-1)}\left(-z\right)^n,
\end{equation}
which does satisfy the recurrence relation
\begin{equation}\label{2phi0recur}
    \left(1-abqz\right)\hat{w}(zq^2)-\left(1-(a+b)zq\right)\hat{w}(zq)-zq\hat{w}(z)=0.
\end{equation}
Its $q$-Borel transform is
\begin{equation}\label{2phi1}
    B_q(t)=\qphi{2}{1}{a,b}{0}{q}{-t}
    =\sum_{n=0}^\infty\frac{\qpr{a}{q}{n}\qpr{b}{q}{n}}{\qpr{q}{q}{n}}\left(-t\right)^n,
    \qquad\qquad |t|<1,
\end{equation}
with recurrence relation
\begin{equation}\label{2phi1recur}
    abt\,B_q(tq^2)-\left(1+(a+b)t\right)B_q(tq)+(1+t)B_q(t)=0.
\end{equation}
To obtain information about the singularities in the complex $t$-plane we use
Heine's second transformation
\cite[\href{http://dlmf.nist.gov/17.6.E7}{17.6.7}]{NIST:DLMF} which can be presented
as
\begin{equation}\label{2phi1cont}
    B_q(t)=\qphi{2}{1}{a,b}{0}{q}{-t}=\frac{\qpr{-bt}{q}{\infty}}{\qpr{-t}{q}{\infty}}
    \qphi{1}{1}{b}{-bt}{q}{-at}.
\end{equation}
Hence, we obtain that $B_q(t)$ has simple poles at $t=-q^{-j}$,
$j=0,1,2,\ldots$, with local behaviour
\begin{equation}\label{btlocal}
    B_q(t)\approx\frac{d_j(a,b)}{1+tq^j},
    \qquad\qquad j=0,1,2,\ldots,
\end{equation}
with $d_{-1}(a,b)=0$, $d_0(a,b)=\frac{\qpr{a}{q}{\infty}\qpr{b}{q}{\infty}}{\qpr{q}{q}{\infty}}$
and from \eqref{2phi1recur} we obtain the recurrence relation
\begin{equation}\label{borelrecur}
    \left(1-q^{-j}\right)d_j=\left(1-(a+b)q^{-j}\right)d_{j-1}+abq^{-j}d_{j-2},\qquad\qquad
    j=1,2,3,\ldots.
\end{equation}
This information can be used in the Cauchy integral representation for the $n$th
coefficient in \eqref{2phi1} and this gives us the large $n$ asymptotic expansion
for the late coefficients:
\begin{equation}\label{2phi1late}
   \frac{\qpr{a}{q}{n}\qpr{b}{q}{n}}{\qpr{q}{q}{n}}\sim \sum_{j=0}^\infty d_j(a,b) q^{jn},
   \qquad\qquad {\rm as}~n\to\infty.
\end{equation}

More importantly, we do obtain $q$-Borel-Laplace transforms of formal series
$\qphi{2}{0}{a,b}{-}{q}{z}$:
\begin{equation}\label{twoBorelLaplace}
    w(z)=
    C_q\int_0^{\infty}
    \qphi{2}{1}{a,b}{0}{q}{-t}\Eq{t/z}\frac{\id t}{t}.
\end{equation}
Let us for the moment assume that $0<q<1$.
This integral representation is well defined for all $\arg z$, that is, for $z$
on the Riemann surface of the complex logarithm. When we push the
contour of integration across the poles at, say $t=\expe^{\pi \iunit} q^{-j}$ we obtain
\begin{equation}\label{2phi0Stokes}
    {w}\left(\expe^{-2\pi \iunit} z\right)-{w}(z)\sim 2\pi\iunit C_q \Eq{-1/z}
    \sum_{n=0}^\infty d_n(a,b) q^{\frac12 n(n+1)}\left(-z\right)^n,
\end{equation}
as $z\to 0$. This is the Stokes phenomenon.

To obtain the exponentially improved expansion we start with the Cauchy integral
representation of a truncated version of \eqref{2phi1}
\begin{equation}\label{2phi1trunc}
    B_q(t)
    =\sum_{n=0}^{N-1}\frac{\qpr{a}{q}{n}\qpr{b}{q}{n}}{\qpr{q}{q}{n}}\left(-t\right)^n
    +\frac{t^N}{2\pi\iunit}\oint_{\{0,t\}}\frac{\tau^{-N}B_q(\tau)}{\tau -t}\id\tau,
\end{equation}
in which the closed contour of integration encloses $0$ and $t$, but not the poles
at $\tau=-q^{-j}$. We `blow-up' the contour of integration to include the contributions
of the simple poles as well and obtain
\begin{equation}\label{2phi1trunc2}
    B_q(t)
    =\sum_{n=0}^{N-1}\frac{\qpr{a}{q}{n}\qpr{b}{q}{n}}{\qpr{q}{q}{n}}\left(-t\right)^n
    +\left(-t\right)^N\sum_{j=0}^\infty \frac{d_j(a,b)q^{(N-1)j}}{t+q^{-j}}.
\end{equation}
We use this result in \eqref{twoBorelLaplace} and obtain for the $q$-Borel-Laplace
transform the re-expansion
\begin{equation}\label{2phi0improved}
    w(z)\sim\sum_{n=0}^{N-1}\frac{\qpr{a}{q}{n}\qpr{b}{q}{n}}{\qpr{q}{q}{n}}
    q^{-\frac12n(n-1)}\left(-z\right)^n
    -\left(-z\right)^{N-1}\sum_{j=0}^\infty d_j(a,b)
    F_q\left(\frac1z;\mytop{N}{q^{-j}}\right),
\end{equation}
as $z\to 0$. In the case that we take $|z|$ small and take the optimal number of terms
$N=\left|\frac{\ln z}{\ln q}\right|+\bigO(1)$, then \eqref{2phi0improved} is an
exponentially improved expansion, that is, we do incorporate exponentially small
terms that do capture the Stokes phenomenon \eqref{2phi0Stokes}.

\section{$\qP1$}\label{Sect:qP1}
We did mention in the introduction that we will study solutions of \eqref{qP1}
with the property $w(z)\sim z$ as $|z|\to 0$.
A formal series expansion of such a solution will be of the form
$\widehat{\w1}(z)=\sum_{n=0}^\infty c_n\left(q^3\right) z^{3n+1}$,
that is, all the coefficients are functions of $q^3$.
For that reason we write our solution of \eqref{qP1} as $w(z)=zv(z^3;q^3)$. This new
function $v(z)=v(z;q)$ does satisfy
\begin{equation}\label{qP1a}
 z v(qz)v^2(z)v(z/q)=v(z)-1.
\end{equation}
The formal solution is of the form
\begin{equation}\label{formal}
    \widehat{\v1}(z)=\sum_{n=0}^\infty c_n(q) z^n,
\end{equation}
with $c_0=1$ and
\begin{equation}\label{crecur}
    c_{n+1}=\sum_{m=0}^n\sum_{k=0}^m\sum_{\ell=0}^{n-m} c_k c_{m-k} c_\ell c_{n-m-\ell}
    q^{k-\ell},\quad n=0,1,2,\cdots.
\end{equation}
In the appendix in Lemma \ref{c_growth} we show that
$\left|c_n\right|\leq \left(\frac{256}{27}\right)^n\left|q\right|^{-\frac12 n(n-1)}$
for $n=0,1,2,\ldots$.
Hence, we introduce the $q$-Borel transform:
\begin{equation}\label{qBorel}
    \b1(t)=\sum_{n=0}^\infty q^{\frac12 n(n-1)}c_n \left(-t\right)^n.
\end{equation}
Note that in the case $0<q<1$ the coefficients $c_n$ are all positive. For that reason
we add an extra factor $\left(-1\right)^n$ in \eqref{qBorel}, and in that way move
the singularities in the complex $t$-plane to the negative real axis. We will have
to add an extra minus sign in the $q$-Laplace transform.

Via Pad\'e approximants, we  observe that for $\b1(t)$ the only singularities in the 
complex $t$-plane are simple poles at $t=-1,-q^{-1}, -q^{-2},\ldots$. 
Thus $(-t;q)_\infty \b1(t)$ is
an entire function. In section \ref{Sect:StPh} we confirm the locations of these
simple poles, because these poles are connected to the `exponentially' small terms
that can be switched on via the Stokes phenomenon.

The $q$-Borel-Laplace transform of $\widehat{\v1}(z)$ is
\begin{equation}\label{qBorelLaplace}
    \v1(z)=C_q\int_0^{\infty}
    \b1(t)\Eq{\expe^{\pi\iunit}t/z}\frac{\id t}{t},
\end{equation}
in which we initially focus on the sector $\arg z\in (0,2\pi)$.

\section{The Stokes multipliers}\label{Sect:StMult}
We observed that $\b1(t)$ is a meromorphic function with simple poles at 
$t=-1,-q^{-1},-q^{-2},\ldots$. It will make sense to call the residues at these 
poles \emph{Stokes multipliers} $K_j(q)$, 
because the local behaviour $\b1(t)\approx\frac{K_j(q)}{t+q^{-j}}$ near $t=-q^{-j}$ 
gives us the asymptotic result
\begin{equation}\label{Kj}
    q^{\frac12 n(n-1)}c_n\sim \sum_{j=0}^\infty K_j(q) q^{j(n+1)},
\end{equation}
as $n\to\infty$.
We can use this result to compute many of the Taylor coefficients of $K_0(q)$ numerically, 
and it seems that $K_0(q)=1+2q+5q^2+10q^3+20q^4+36q^5+65q^6+110 q^7+185q^8+\cdots$, which suggests that we have
\begin{equation}\label{K0}
    K_0(q)=\frac{1}{\left(q;q\right)_\infty^2}.
\end{equation}
Taking values for $q$ in the interval $(\frac12,1)$ does numerically verify this identification.
In Lemma \ref{dinfo} we do show that $d_n(q)=q^{\frac12 n(n-1)}c_n$
is a polynomial in $q$ of degree $n(n-1)$ with positive integer coefficients. The first $n$ coefficients are the same for 
$d_n, d_{n+1}, d_{n+2}, \ldots$. For example,
$d_5=1+2q+5q^2+10q^3+20q^4+32q^5+\cdots +q^{20}$.
In Lemma \ref{localStokes} we do show that by computing $d_n$, we do compute the first $n$
Taylor coefficients of $K_0(q)$. These calculations also suggest that \eqref{K0} is correct.

When we use \eqref{K0} in \eqref{Kj} we can compute the Taylor coefficients of $K_1(q)$
numerically, and it seems that
\begin{equation}\label{K1}
    \frac{K_1(q)}{K_0(q)}=-q^{-1}\left(4 +\frac{5q}{1-q}\right).
\end{equation}
Next, we use \eqref{K0} and \eqref{K1} in \eqref{Kj} we can compute the Taylor coefficients of $K_2(q)$
numerically, and it seems that
\begin{equation}\label{K2}
    \frac{K_2(q)}{K_0(q)}=-q^{-4}\left(4-q-10q^2-\frac{q^3\left(24+10q-9q^2\right)}{\left(1-q\right)\left(1-q^2\right)}\right).
\end{equation}
Again, taking values for $q$ in the interval $(\frac12,1)$ does numerically verify these identifications.

Similarly (but a lot of work)
\begin{gather}
\begin{split}\label{K3}
    \frac{K_3(q)}{K_0(q)}&=-q^{-9}\Biggl(
    4+8q+2q^2-10q^3-23q^4-15q^5+6q^6\Biggr.\\
    &\qquad\qquad\Biggl.+\frac{q^7\left(52+60q+38q^2-12q^3-20q^4+7q^5\right)}{
    \left(1-q\right)\left(1-q^2\right)\left(1-q^3\right)}\Biggr),
\end{split}
\end{gather}
and
\begin{equation}\label{K4}
    \frac{K_4(q)}{K_0(q)}=-q^{-16}\left(
    4+8q+20q^2+22q^3+5q^4-34q^5-71q^6-74q^7-51q^8
    +\cdots\right).
\end{equation}
In \S\ref{Sect:Trans} we will show that these  expressions for $K_j(q)/K_0(q)$, $j=1,2,3,4$
are correct.

\section{The Stokes phenomenon and exponentially improved asymptotics}\label{Sect:StPh}
From the details above it should be obvious that both $\v1(z)$ and 
$\v1\left(\expe^{2\pi\iunit}z\right)$ have the same asymptotic
expansion \eqref{formal}. Hence, the difference has to be exponentially small.
Starting with $\arg z\in (0,2\pi)$ we can combine the $q$-Borel-Laplace
integral representations of both functions into the same integral, but with a contour
of integration that encircles the simple poles of $\b1 (t)$. In this way we obtain
the Stokes phenomenon
\begin{gather}
\begin{split}\label{Stokesph}
    \v1(z)-\v1\left(\expe^{2\pi\iunit}z\right)&\sim 
    2\pi\iunit C_q\sum_{j=0}^\infty q^j K_j(q)\Eq{q^{-j} z^{-1}}\\
    &=2\pi\iunit C_q\Eq{z^{-1}}\sum_{j=0}^\infty q^{\frac12 j(j+3)} K_j(q)z^{j}.
\end{split}
\end{gather}
One important observation is the following. If the singularities $t=-q^{-j}$ would 
have been singularities in a normal Borel plane, then they would contribute at 
different exponential levels, however, here it should be obvious from the right-hand 
side of \eqref{Stokesph} that they contribute at the same exponential
level. It seems from \S\ref{Sect:StMult} that $K_j(q)=q^{-j^2}\bigO(1)$, as $j\to\infty$. 
If that is the case then the sums in \eqref{Stokesph} will be divergent.

With an analysis that is very similar to the one at the end of \S\ref{S:2phi0} we
obtain for the $q$-Borel-Laplace transform the re-expansion
\begin{equation}\label{Painlimproved}
    \v1(z)\sim\sum_{n=0}^{N-1}c_n z^n
    -z^{N-1}\sum_{j=0}^\infty q^j K_j(q)
    F_q\left(\frac{-1}{z};\mytop{N}{q^{-j}}\right),
\end{equation}
as $z\to 0$. In the case that we take $|z|$ small and take the optimal number of terms
$N=\left|\frac{\ln z}{\ln q}\right|+\bigO(1)$, then \eqref{Painlimproved} is an
exponentially improved expansion, that is, we do incorporate exponentially small
terms that do capture the Stokes phenomenon \eqref{Stokesph}.

\section{Transseries analysis}\label{Sect:Trans}
We will have a look at exponentially small terms by using a transseries $v(z)=v_0(z)+v_1(z)+v_2(z)+\cdots$
in \eqref{qP1a}, in which $v_0(z)\sim \widehat{\v1}(z)$, see \eqref{formal}, 
$v_1(z)$ is exponentially small, $v_2(z)$ is double exponentially small, and so on. 
For $v_1(z)$ we obtain the linear $q$-difference equation
\begin{equation}\label{w1equation}
    v_0(qz)v_0^2(z)v_1(z/q)+2v_0(qz)v_0(z)v_0(z/q)v_1(z)+v_0(z/q)v_0^2(z)v_1(qz)=z^{-1}v_1(z).
\end{equation}
Using the fact that $v_0(z)\sim 1$ we have approximately
\begin{equation}\label{w1equation2}
    v_1(z/q)+2v_1(z)+v_1(qz)\approx z^{-1}v_1(z).
\end{equation}
If \eqref{Stokesph} is correct then we should have
\begin{gather}
\begin{split}\label{w1equation3}
    v_1(z)\approx C\Eq{z^{-1}}\Longrightarrow
   v_1(z/q)&\approx \frac{C}{z}\Eq{z^{-1}},~~
     v_1(qz)\approx Cqz\Eq{z^{-1}},
\end{split}
\end{gather}
with $C=2\pi\iunit C_{q}K_0\left(q\right)$.
We observe that we have $v_1(z/q)\approx z^{-1}v_1(z)$,
and these two terms are the dominant terms in \eqref{w1equation2}, as $z\to0$.

We can also use this analysis to confirm that in the $q$-Borel plane we will have simple poles at $t=q^{-j}$: We are looking
for exponentially-small solutions of \eqref{w1equation2}. The principle of dominant balance does show us that
the only acceptable balance is $v_1(z/q)= z^{-1}v_1(z)$, and this has the general solution
$v_1(z)=P(z)/\theta_{q}\left(z^{-1}\right)$, 
with $P(z)$ being $q$-periodic, that is, $P(qz)=P(z)$. Recall \eqref{qperiodic}.
Once we have this dominant behaviour we obtain a formal solution of \eqref{w1equation} of the form
of the right-hand side of \eqref{Stokesph} with undetermined, and potentially
$q$-periodic, coefficients $K_j(q)$.
Hence, the $\b1(t)$ in integral representation \eqref{qBorelLaplace} should have simple poles at $t=q^{-j}$.

When we take for $v_1(z)$ the right-hand side of \eqref{Stokesph}, use it in \eqref{w1equation},
and compare the $z$-Taylor coefficients we obtain the equations
\begin{gather}
\begin{split}\label{KjEquations}
   \left(q+4\right) K_{0}+q\left(1-q\right) K_{1}&=0,\\
   \left(q^{4}+2 q^{3}+8 q^{2}+7 q+4\right) K_{0}+q^2\left(2+3 q\right) K_{1}+q^{4}\left(1-q^{2}\right) K_{2}&=0,\\
   \left(q^{9}+2 q^{8}+5 q^{7}+12 q^{6}+19 q^{5}+29 q^{4}+26 q^{3}+18 q^{2}+8 q+4\right) K_{0}\qquad\qquad&\\
   +q^{3}\left(2+7 q+7 q^{2}+4 q^{3}+2 q^{4}\right) K_{1}+q^{6}\left(2+q+2 q^{2}\right) K_{2}+q^{9}\left(1-q^{3}\right) K_{3}&=0,
\end{split}
\end{gather}
confirming \eqref{K1}, \eqref{K2} and \eqref{K3}. 
Hence, all the $K_j(q)$, $j=1,2,3,\ldots$, can be computed via this method.

To go a bit deeper in the exponential asymptotics, and to determine the singularities
that appear on the boundary of the sector of validity of the dominant asymptotic
expansion we present the transseries as
\begin{equation}\label{transseries}
    v(z)=\sum_{n=0}^\infty C^n(z)v_n(z),
\end{equation}
where the free $q$-periodic function $C(z)$ keeps track of the exponential scale.
We substitute this into the original equation \eqref{qP1a} and at level $C^1(z)$ we obtain
equation \eqref{w1equation}. Using $v_0(z)\sim 1$ we obtain that the dominant terms
in \eqref{w1equation} are $v_1(z/q)\sim z^{-1} v_1(z)$. At level $C^2(z)$ we use these
facts for $v_0$ and $v_1$ and obtain the dominant terms $v_2(z/q)\sim -2z^{-1} v_1^2(z)$.
Via induction it can be shown, using just the facts for $v_0$ and $v_1$, that at
level $C^n(z)$ we have the dominant terms
\begin{equation}\label{wn}
    v_n(z/q)\sim \left(-1\right)^{n-1}nz^{-1}v_1^n(z)
    \quad\Longrightarrow\quad v_n(z)\sim n\left(-qz\right)^{n-1}v_1^n(z),
    \quad n=1,2,3\ldots.
\end{equation}
Using this result in \eqref{transseries} we can resum the transseries as
\begin{equation}\label{resum}
    v(z)\sim v_0(z)+\frac{C(z)v_1(z)}{\left(1+C(z)q z v_1(z)\right)^2}.
\end{equation}
Hence, we predict that there will be poles where $1+C(z)q z v_1(z)=0$.

\section{The order 1 solution and the distribution of poles and zeros}\label{Sect:PolesZeros}
There are also regular solutions near $z=0$. When we substitute
\begin{equation}\label{regular}
    w(z)=\sum_{n=0}^\infty e_n z^n,
\end{equation}
into \eqref{qP1} then we obtain for the coefficients $e_0^3=1$, $e_1=\frac{-1}{q+1+q^{-1}}$,
and the recurrence relation
\begin{equation}\label{regularcoeff}
    \left(q+1+q^{-1}\right)e_n=-e_0^2\sum_{m=1}^{n-1} \left(1+q^{n-2m}\right)e_me_{n-m}
    -\sum_{m=1}^{n-1}\sum_{k=0}^{m}\sum_{\ell=0}^{n-m}q^{n-m-2\ell}
    e_k e_{m-k} e_\ell e_{n-m-\ell}.
\end{equation}
Expansion \eqref{regular} seems to be a convergent expansion for a meromorphic solution.

\begin{figure}[h]
\centering\includegraphics[width=0.4\textwidth]{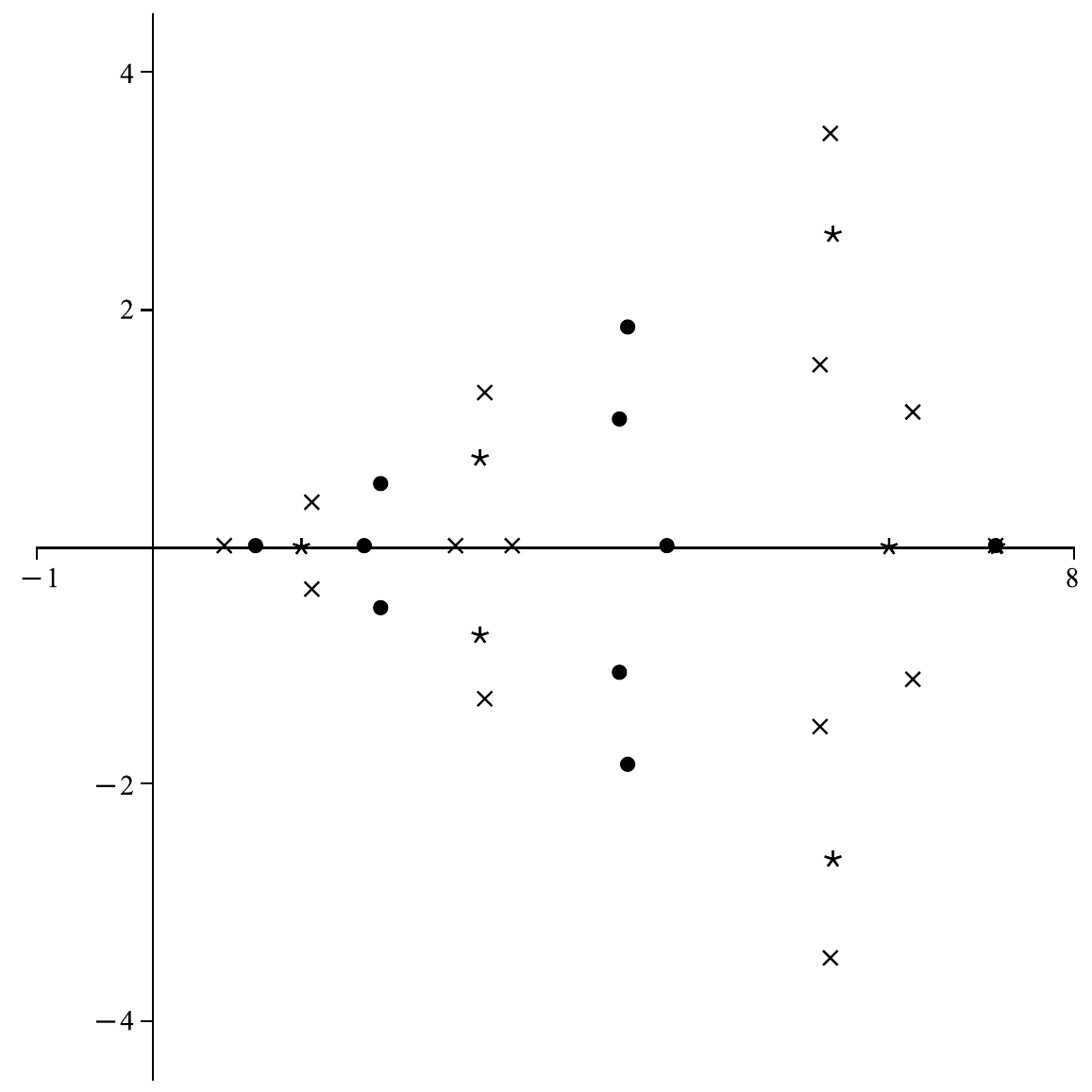}%
\caption{The case $e_0=1$, $q=\frac7{10}$. 
The $\star$ are the locations of the double poles,
the $\bullet$ are the simple zeros, and the $\times$ are stationary points, that is, 
solutions of $w(z)=z$.}
\label{fig1}
\end{figure}

For the case $e_0=1$ and $q=\frac7{10}$ we have computed the Pad\'e approximant of order
[40/40]. In Figure \ref{fig1} we demonstrate that if $w(z)$ has a pole at $z=z_p$,
which seems to be a double pole, then $w(z)$ has simple zeros at $z=q^{\pm1}z_p$,
and stationary points ($w(z)=z$) at $z=q^{\pm2}z_p$. This follows from \eqref{qP1}:
If $w(z)$ has a stationary point at $z=z_0$, then it follows from \eqref{qP1} that either 
$z=qz_0$ or $z=q^{-1}z_0$ is a zero. Say that the zero is at $z=qz_0$, then by keep on
using \eqref{qP1} we obtain that $z=q^2z_0$ is a `pole', $z=q^3z_0$ is another zero,
and $z=q^4z_0$ is another stationary point. In fact, a simple local analysis does show us
that if near the `pole' we have $w(z)\approx -C_1^2\left(z-z_p\right)^{-2\alpha}$, with
$\Re\alpha>0$, then near the zeros we have
$w(z)\approx \pm\frac{q^{\pm\frac12\pm\alpha}}{C_1}\left(z-q^{\mp1}z_p\right)^{\alpha}$.
In the example above we seem to have $\alpha=1$, and for the nearest pole we have
$z_p\approx 1.2946$.

\section*{Acknowledgements}
The authors want to thank the Isaac Newton Institute for Mathematical Sciences for support during the program ‘Applicable resurgent asymptotics: towards a universal theory’ supported by EPSRC Grant No. EP/R014604/1. AOD's research was supported by a research Grant 60NANB20D126 from the National Institute of Standards and Technology, and parts of this research was conducted while visiting the Okinawa Institute of Science and Technology (OIST) through the Theoretical Sciences Visiting Program (TSVP).

\appendix
\section{Minor lemmas}
\begin{lemma}\label{c_growth}
For the sequence $\{c_n\}$ defined by $c_0=1$ and recurrence relation \eqref{crecur} 
we have $\left|c_n\right|\leq \left(\frac{256}{27}\right)^n\left|q\right|^{-\frac12 n(n-1)}$
for $n=0,1,2,\ldots$ and $|q|\leq1$.
\end{lemma}
\begin{proof}
We take $c_n=q^{-\frac12 n(n-1)} d_n$, and obtain the recurrence relation
\begin{equation}\label{drecur}
    d_{n+1}=\sum_{m=0}^n\sum_{k=0}^m\sum_{\ell=0}^{n-m} d_k d_{m-k} d_\ell d_{n-m-\ell}
    q^{p(n,m,k,\ell)},\quad n=0,1,2,\ldots,
\end{equation}
with
\begin{equation}\label{pnmkl}
    p(n,m,k,\ell)=m(n-m+1)+k(m-k+1)+\ell(n-m-\ell+1).
\end{equation}
Note that all the terms on the right-hand side of \eqref{pnmkl} are non-negative.
Hence, $p(n,m,k,\ell)\geq0$. 
Thus $\left|q^{p(n,m,k,\ell)}\right|\leq1$.
It follows that for all $n$ we have $|d_n|\leq \widetilde{d}_n$,
where the sequence $\{\widetilde{d}_n\}$ is defined by $\widetilde{d}_0=1$ and 
the recurrence relation
\begin{equation}\label{dtrecur}
    \widetilde{d}_{n+1}=\sum_{m=0}^n\sum_{k=0}^m\sum_{\ell=0}^{n-m} \widetilde{d}_k 
    \widetilde{d}_{m-k} \widetilde{d}_\ell \widetilde{d}_{n-m-\ell},\quad n=0,1,2,\ldots.
\end{equation}
Observe that $\widetilde{v}(z)=\sum_{n=0}^\infty \widetilde{d}_n z^{n}$ is a solution of
$z\widetilde{v}^4(z)=\widetilde{v}(z)-1$.
This can be solved via a Fox--Wright function giving
\begin{equation}\label{FW}
\widetilde{v}(z) = \sum_{n = 0}^\infty\frac{\Gamma (4n + 1)}{\Gamma(3n + 2)}\frac{z^n}{n!}
\end{equation}
(cf. \cite[Eq. (9.17)]{Belkic2019}). Thus
\begin{gather}\label{d0-d4}
    \begin{split}
\left|d_n (q)\right| \le \widetilde{d}_n=\frac{\Gamma (4n + 1)}{\Gamma (3n + 2)n!} & \le
\frac{ (4n)^{4n + 1/2} \expe^{ - 4n} \expe }{(3n + 1)(3n)^{3n + 1/2} \expe^{ - 3n} \sqrt {2\pi } n^{n + 1/2} \expe^{ - n} \sqrt {2\pi } } \\ & = \frac{\expe}{\pi \sqrt 3 }\frac{1}{\sqrt n (3n + 1)}\left( \frac{256}{27} \right)^n  \le \left( \frac{256}{27} \right)^n 
    \end{split}
\end{gather}
using
\begin{equation}\label{fact}
n^{n + 1/2} \expe^{ - n} \sqrt {2\pi }  \le n! \le  n^{n + 1/2} \expe^{ - n} \expe.
\end{equation}

\end{proof}


\begin{lemma}\label{dinfo}
The coefficients defined in \eqref{drecur} are
\begin{gather}\label{d0-d4}
    \begin{split}
        d_0=d_1&=1,\\
        d_2&=1+2q+q^2,\\
        d_3&=1+2q+5q^2+6q^3+5q^4+2q^5+q^6,\\
        d_4&=1+2q+5q^2+10q^3+16q^4+23q^5+26q^6+23q^7+16q^8+10q^9+5q^{10}+2q^{11}+q^{12},\\
        d_5&=1+2q+5q^2+10q^3+20q^4+32q^5+52q^6+75q^7+101q^8+\cdots+5q^{18}+2q^{19}+q^{20}.
    \end{split}
\end{gather}
In general, $d_n(q)$ is a polynomial of degree $n(n-1)$ with positive integer coefficients, and the first $n$
coefficients are the same for $d_n, d_{n+1}, d_{n+2}, \ldots$. Hence, $d_n(0)=1$. We also have the symmetry
$d_n(q)=q^{n(n-1)}d_n(1/q)$.
\end{lemma}
\begin{proof}
The statement about the degree and the positive integer coefficients follows, via induction, from recurrence relation
\eqref{drecur}. Let $n\geq2$. We do observe from \eqref{pnmkl} that $p(n,m,k,\ell)=0$ only when $m=k=\ell=0$,
and the smallest non-zero value for $p(n,m,k,\ell)$ is $n$, and this happens when $(m,k,\ell)=(1,0,0)$, 
$(m,k,\ell)=(0,0,1)$, $(m,k,\ell)=(n,0,0)$, and $(m,k,\ell)=(0,0,n)$. Hence, \eqref{drecur} can be presented as
\begin{equation}\label{drecur2}
    d_{n+1}-d_n=q^n\left(2d_{n-1}+2d_n\right)+\cdots,\qquad \Longrightarrow\qquad
    d_{n+1}-d_n=4q^n+h.o.t.
\end{equation}
\end{proof}

\begin{lemma}\label{localStokes}
For coefficients defined in \eqref{drecur} we use the notation
\begin{equation}\label{nm}
    d_n(q)=\sum_{m=0}^{n(n-1)}d_{n,m}q^m.
\end{equation}
Then for $|q|\leq\frac{27}{256}$ the Stokes multiplier $K_0(q)$ has the expansion
\begin{equation}\label{K0expansion}
    K_0(q)=1+\sum_{n=1}^{\infty}d_{n+1,n}q^n,
\end{equation}
and we have
\begin{equation}\label{ResurgenceRigorious}
    c_n(q)\sim q^{-\frac12n(n-1)}K_0(q),
\end{equation}
as $n\to\infty$.
\end{lemma}
\begin{proof}
In Lemma \ref{c_growth} we showed that $\left|d_n (q)\right| \le M^n$ with $M=\frac{256}{27}$, and in Lemma \ref{dinfo} we showed that 
$d_{n + 1} (q) - d_n (q) = 4q^n  + \bigO(q^{n + 1} )$
for $n\ge 2$. Thus by Cauchy's formula
\begin{equation}\label{dCauchy}
d_{n + 1} (q) - d_n (q) = \frac{q^n}{2\pi \iunit}\oint_{|t| = 1} \frac{d_{n + 1} (t) - d_n (t)}{(t - q)t^n}\id t 
\end{equation}
for $|q|<1$ and $n\ge 2$. Then
\begin{gather}\label{dCauchy2}
\begin{split}
| d_{n + 1} (q) - d_n (q)| &\le \frac{\left|q\right|^n }{2\pi }\oint_{|t| = 1} 
\frac{|d_{n + 1} (t) - d_n (t)|}{|t - q|\left|t\right|^n }|\id t|\\
&\le 2M^{n + 1} \frac{\left|q\right|^n}{2\pi }\oint_{|t| = 1} \frac{|\id t|}{|t - q|}  = 2M\frac{\left(M|q|\right)^n }{1 - |q|}
\end{split}
\end{gather}
for $|q|<1$ and $n\ge 2$. Fix $\rho\in(0,1)$ and assume that $|q| \le \frac{\rho }{M} < \frac{1}{M}
$ and $m>n\ge 2$. Then
\begin{gather}\label{dCauchy3}
\begin{split}
|d_m (q) - d_n (q)| &\le \sum_{k = n}^{m - 1} |d_{k + 1} (q) - d_k (q)|  \le \frac{2M}{1-|q|}\sum_{k = n}^{m - 1} \left(M|q|\right)^k \\
&= \frac{2M}{1 - |q|}\frac{\left(M|q|\right)^n  - \left(M|q|\right)^m }{1 - M|q|} \le \frac{2M^2 }{M- 1}\frac{\rho ^n }{1-\rho }.
\end{split}
\end{gather}
Thus $\{ d_n (q)\}$ is uniformly Cauchy on compact subsets of $|q| < \frac{1}{M}
$. Therefore, it has a limit that is analytic on $|q| < \frac{1}{M}$. We can call it $K_0(q)$. Taking $m\to  \infty$ in the above, we have
\begin{equation}\label{dCauchy4}
|K_0 (q) - d_n (q)| \le \frac{2M^2}{M - 1}\frac{\left(M|q|\right)^n }{1 - \rho }
\end{equation}
for $|q| \le \frac{\rho }{M} < \frac{1}{M}
$ and $n\ge 2$. We can use this and induction to show that both \eqref{K0expansion} 
and \eqref{ResurgenceRigorious} hold.
It would be nice to show that $K_0(q)$ is actually analytic in $|q|<1$.
\end{proof}

\section{A simple nonlinear equation}\label{AppRiccati}
The nonlinear equation
\begin{equation}\label{simpleeq}
    w(z)w(z/q)=w(z)-z,
\end{equation}
is a much simpler version of \eqref{qP1}, and a formal divergent solution
is of the form $\displaystyle\widehat{w}(z)=\sum_{n=1}^\infty \widetilde{c_n}z^n$.
We use the `Riccati'-ansatz $\displaystyle w(z)=z\frac{h(z)}{h(qz)}$
and obtain for $h(z)$ the linear equation
$h(z)-h(qz)=\frac{z}{q} h(z/q)$. In this way we obtain the solution
\begin{equation}\label{simpleRic}
    w(z)=z\frac{\qphi{2}{0}{0,0}{-}{q}{-z/q}}{\qphi{2}{0}{0,0}{-}{q}{-z}},
\end{equation}
which can be used to prove that
\begin{equation}
    \widetilde{c_n}\sim\frac{q^{-\frac12 n(n-1)}}{\left(q;q\right)_\infty},
\end{equation}
as $n\to\infty$.

We try a similar `Riccati'-ansatz for our main equation \eqref{qP1a}
via $\displaystyle v(z)=\frac{h(z)}{h(qz)}$ and obtain the equation
\begin{equation}\label{Riccati}
    zh(z)h(z/q)=h(z)h(zq^2)-h(zq)h(zq^2).
\end{equation}
(Is it possible to make one more step and obtain a linear equation?)
We are interested in the formal solution
$\displaystyle\widehat{h}(z)=\sum_{n=0}^\infty e_nz^n$.
The coefficients have the recurrence relation
\begin{equation}\label{RiccatiCoeff}
    \left(1-q^n\right) e_n=
    \sum_{m=0}^{n-1}e_m e_{n-m-1} q^{-m} -
    \sum_{m=1}^{n-1}e_m e_{n-m}\left(q^{2m}-q^{n+m}\right),
\end{equation}
with $e_0$ arbitrary. We take $e_0=1$. These $e_n$ have the same behaviour as
our main $c_n$, but the recurrence relation \eqref{RiccatiCoeff} seems
much simpler than \eqref{crecur}. We take $e_n=q^{-\frac12 n(n-1)} f_n$.
These $f_n$ are very similar to $d_n$: They do satisfy
\begin{equation}\label{fneq}
    \sum_{m=0}^{n-1}f_m f_{n-m-1} q^{(m+1)(n-m-1)} =
    \sum_{m=0}^{n-1}f_m f_{n-m} q^{m(n-m+2)}\left(1-q^{n-m}\right),
\end{equation}
and again $f_5=1+2q+5q^2+10q^3+20q^4+34q^5+\cdots$. Ideally we show now that
\begin{equation}\label{fnasymp}
    f_n\sim\frac1{\left(q;q\right)_\infty^2},\qquad\qquad{\rm or}\qquad\qquad
    f_n=\frac1{\left(q;q\right)_n^2}+\bigO\left(q^n\right),
\end{equation}
as $n\to\infty$. Taylor series expansions, and numerics does show that \eqref{fnasymp}
is correct.


\begin{thebibliography}{10}

\bibitem{Adachi2019}
{\sc S.~Adachi}, {\em The {$q$}-{B}orel sum of divergent basic hypergeometric
  series {$_r\varphi_s(a;b;q,x)$}}, SIGMA Symmetry Integrability Geom. Methods
  Appl., 15 (2019), pp.~Paper No. 016, 12.

\bibitem{Belkic2019}
{\sc D.~Belkić}, {\em All the trinomial roots, their powers and logarithms
  from the {L}ambert series, {B}ell polynomials and {F}ox–{W}right function:
  illustration for genome multiplicity in survival of irradiated cells}, J.
  Math. Chem., 57 (2019), pp.~59--106.

\bibitem{Berry1988}
{\sc M.~V. Berry}, {\em Uniform asymptotic smoothing of {S}tokes's
  discontinuities}, Proc. Roy. Soc. London Ser. A, 422 (1989), pp.~7--21.

\bibitem{BGT19}
{\sc G.~Bonelli, A.~Grassi, and A.~Tanzini}, {\em Quantum curves and
  {$q$}-deformed {P}ainlev\'{e} equations}, Lett. Math. Phys., 109 (2019),
  pp.~1961--2001.

\bibitem{NIST:DLMF}
{\em {\it NIST Digital Library of Mathematical Functions}}.
\newblock \url{https://dlmf.nist.gov/}, Release 1.1.12 of 2023-12-15.
\newblock F.~W.~J. Olver, A.~B. {Olde Daalhuis}, D.~W. Lozier, B.~I. Schneider,
  R.~F. Boisvert, C.~W. Clark, B.~R. Miller, B.~V. Saunders, H.~S. Cohl, and
  M.~A. McClain, eds.

\bibitem{J15}
{\sc N.~Joshi}, {\em Quicksilver solutions of a {$q$}-difference first
  {P}ainlev\'{e} equation}, Stud. Appl. Math., 134 (2015), pp.~233--251.

\bibitem{Mitschi2016}
{\sc C.~Mitschi and D.~Sauzin}, {\em Divergent series, summability and
  resurgence. {I}}, vol.~2153 of Lecture Notes in Mathematics, Springer,
  [Cham], 2016.
\newblock Monodromy and resurgence, With a foreword by Jean-Pierre Ramis and a
  preface by \'{E}ric Delabaere, Mich\`ele Loday-Richaud, Claude Mitschi and
  David Sauzin.

\bibitem{N10}
{\sc S.~Nishioka}, {\em Transcendence of solutions of {$q$}-{P}ainlev\'{e}
  equation of type {$A^{(1)}_7$}}, Aequationes Math., 79 (2010), pp.~1--12.

\bibitem{OD94}
{\sc A.~B. Olde~Daalhuis}, {\em Asymptotic expansions for {$q$}-gamma,
  {$q$}-exponential, and {$q$}-{B}essel functions}, J. Math. Anal. Appl., 186
  (1994), pp.~896--913.

\bibitem{OD1998}
\leavevmode\vrule height 2pt depth -1.6pt width 23pt, {\em Hyperterminants.
  {II}}, J. Comput. Appl. Math., 89 (1998), pp.~87--95.

\bibitem{S01}
{\sc H.~Sakai}, {\em Rational surfaces associated with affine root systems and
  geometry of the {P}ainlev\'{e} equations}, Comm. Math. Phys., 220 (2001),
  pp.~165--229.

\bibitem{Tahara2017}
{\sc H.~Tahara}, {\em {$q$}-analogues of {L}aplace and {B}orel transforms by
  means of {$q$}-exponentials}, Ann. Inst. Fourier (Grenoble), 67 (2017),
  pp.~1865--1903.

\bibitem{Zhang2000}
{\sc C.~Zhang}, {\em Transformations de {$q$}-{B}orel-{L}aplace au moyen de la
  fonction th\^{e}ta de {J}acobi}, C. R. Acad. Sci. Paris S\'{e}r. I Math., 331
  (2000), pp.~31--34.

\bibitem{Zhang2002}
\leavevmode\vrule height 2pt depth -1.6pt width 23pt, {\em Une sommation
  discr\`ete pour des \'{e}quations aux {$q$}-diff\'{e}rences lin\'{e}aires et
  \`a coefficients analytiques: th\'{e}orie g\'{e}n\'{e}rale et exemples}, in
  Differential equations and the {S}tokes phenomenon, World Sci. Publ., River
  Edge, NJ, 2002, pp.~309--329.

\end{thebibliography}

\end{document}